\documentclass[11pt, reqno]{amsproc} 

\usepackage{changes}
\usepackage{geometry}
\usepackage{color}
\usepackage{verbatim}
\usepackage[utf8x]{inputenc}
\usepackage{t1enc}
\usepackage[english]{babel}

\usepackage{amsmath,amssymb,amsthm}
\usepackage{mathrsfs,esint}

\usepackage{graphicx,wrapfig}
\usepackage[format=hang]{caption}

\newcommand*{\R}{{\mathbb R}}

\newcommand*{\N}{{\mathbb N}}

\newcommand*{\eps}{\varepsilon}

\newcommand*{\Om}{\Omega}

\newcommand{\EC}{\mathcal{E}}

\newcommand*{\pip}{\varphi}
\newcommand{\uB}{\bar{u}}

\newcommand{\xB}{\bar{x}}
\newcommand{\yB}{\bar{y}}
\newcommand{\zB}{\bar{z}}
\newcommand{\muB}{\mu_r}

\newcommand{\dB}{d_r}

\providecommand*{\vint}[1]{\mathchoice
          {\mathop{\vrule width 5pt height 3 pt depth -2.5pt
                  \kern -9pt \kern 1pt\intop}\nolimits_{\kern -5pt{#1}}}
          {\mathop{\vrule width 5pt height 3 pt depth -2.6pt
                  \kern -6pt \intop}\nolimits_{\kern -3pt{#1}}}
          {\mathop{\vrule width 5pt height 3 pt depth -2.6pt
                  \kern -6pt \intop}\nolimits_{\kern -3pt{#1}}}
          {\mathop{\vrule width 5pt height 3 pt depth -2.6pt
                  \kern -6pt \intop}\nolimits_{\kern -3pt{#1}}}}
\newcommand*{\jint}{\fint}
\DeclareMathOperator{\Lip}{Lip}

\DeclareMathOperator{\dist}{dist}

\DeclareMathOperator{\rad}{rad}

\numberwithin{equation}{section}
\theoremstyle{plain}
\newtheorem{thm}[equation]{Theorem}

\newtheorem{prop}[equation]{Proposition}

\newtheorem{lem}[equation]{Lemma}

\theoremstyle{definition}

\newtheorem{defn}[equation]{Definition}

\newtheorem{remark}[equation]{Remark}

\begin{document}

\title[Construction of Dirichlet form and solving Dirichlet problems]
{Construction of a Dirichlet form on metric measure spaces of controlled geometry} 
\author{Almaz Butaev}
\address{Department of Mathematical Sciences, P.O.~Box 210025, University of Cincinnati, Cincinnati, OH~45221-0025, U.S.A.
and
Department of Mathematics and Statistics, University of the Fraser Valley, Abbotsford, BC, V2S 7M8, Canada.
}
\email{Almaz.Butaev@ufv.ca}
\author{Liangbing Luo}
\address{Department of Mathematical Sciences, Lehigh University, Bethlehem, PA 18015, U.S.A.}
\email{lil522@lehigh.edu}
\author{Nageswari Shanmugalingam}
\address{Department of Mathematical Sciences, P.O.~Box 210025, University of Cincinnati, Cincinnati, OH~45221-0025, U.S.A.}
\email{shanmun@uc.edu}
\thanks{
 N.S.'s work is partially supported by the NSF (U.S.A.) grant DMS~\#2054960. This work was begun during the residency
 of L.L.~and N.S.~at the 
 Mathematical Sciences Research Institute (MSRI, Berkeley, CA) as members of  the program
\emph{Analysis and Geometry in Random Spaces} which is 
 supported by the National Science Foundation (NSF U.S.A.) under Grant No. 1440140, during Spring 2022. 
 They thank MSRI for its kind hospitality. The authors thank Patricia Alonso-Ruiz and Fabrice Baudoin for sharing with us their early 
 manuscript~\cite{A-RB}, which helped us
work out the nature of  the tool of Mosco convergence.}

\begin{abstract}
Given a compact doubling metric measure space $X$ that supports a $2$-Poincar\'e inequality, we construct a Dirichlet form on 
$N^{1,2}(X)$ that is comparable to the upper gradient energy form on $N^{1,2}(X)$. Our approach is based on the approximation 
of $X$ by a family of graphs that is doubling and supports a $2$-Poincar\'e inequality (see \cite{GillLopez}). 
We construct a bilinear form on $N^{1,2}(X)$ using the Dirichlet form on the graph. 
We show that the $\Gamma$-limit $\mathcal{E}$ of this family of bilinear forms (by taking a subsequence) exists and that 
$\mathcal{E}$ is a Dirichlet form on $X$. Properties of $\mathcal{E}$ are established. Moreover, we prove that $\mathcal{E}$ has 
the property of matching boundary values on a domain $\Omega\subseteq X$. This 
construction makes it possible to approximate harmonic functions (with respect to the Dirichlet form $\EC$) on a domain in $X$ with
a prescribed Lipschitz boundary data via a numerical scheme dictated by the approximating Dirichlet forms, which are discrete objects.
\end{abstract}
\maketitle
\noindent
    {\small \emph{Key words and phrases}: Dirichlet form, $N^{1,2}(X)$, graph approximation, $\Gamma$-convergence, boundary values,
    doubling spaces, Poincar\'e inequality, Lipschitz approximations.
}

\medskip

\noindent
    {\small Mathematics Subject Classification (2020):
Primary: 31C25.
Secondary: 46E35, 49Q20, 65N55.
}

\section{Introduction}

The idea of approximating a compact doubling metric space by graphs now is well-established, with the work of
Christ~\cite{ChristCubes} containing a prototype of such an idea. This idea blossomed into the construction of
hyperbolic fillings of compact doubling spaces to obtain a Gromov hyperbolic space whose boundary is
(homeomorphic to) the given space (see for instance~\cite{BonkSak, BourPaj, Carrasco, GillLopez}). Under the
additional condition that the compact doubling metric space is equipped with a doubling measure that supports a 
Poincar\'e inequality, much can be said about the approximating graphs, see~\cite{GillLopez}. In the non-smooth 
setting of metric measure spaces equipped with a doubling measure supporting a Poincar\'e inequality, 
it is now known that there is a rich theory of harmonic functions (see for instance~\cite{BB-book, BBS-Dirichlet, Sh-Illinois}).
What is missing in the current literature is a construction of a Dirichlet form which  
is compatible with the upper gradient structure on the underlying metric
measure space and, in addition, a way of approximating the Dirichlet form via its discrete couterparts  
defined on graphs approximating 
the metric space so that solutions to a Dirichlet boundary value problem for functions that are harmonic with respect to the limit
Dirichlet form can be approximated by solutions to Dirichlet boundary value problem for the 
approximating sequence of graph-Dirichlet forms. The goal of this present note is to demonstrate one way of
constructing such an approximation.

The notion of Dirichlet forms has its roots in the work of Beurling and Deny~\cite{BD1, BD2}. The theory of
Dirichlet forms has been fleshed out in the textbook~\cite{FOT}, and an excellent exposition of
the notions of $\Gamma$-convergence 
and Mosco-convergence of Dirichlet forms can be found for instance in~\cite{Braides,DeGiorgiSpag, KuwaeShioya03,
KuwaeShioya08, Mosco}. Of these expositions, the works of~\cite{Braides, DeGiorgiSpag, Mosco} consider sequences of Dirichlet forms
on Euclidean spaces, while the works~\cite{KuwaeShioya03, KuwaeShioya08} consider sequences of Dirichlet forms,
each defined on a (potentially) different metric measure space, converging in a sense akin to measured Gromov-Hausdorff
convergence (see~\cite{Chee} for more on Gromov-Hausdorff convergence) of Dirichlet forms.

The paper~\cite{KuwaeShioya03} discusses convergence of graphs to graphs and the 
convergence of corresponding
Dirichlet forms, while~\cite{KuwaeShioya08} studies Gromov-Haudorff convergence of function classes
and Dirichlet forms. 
The papers~\cite{KumagaiSturm, A-RB} construct a Dirichlet form that is equivalent to the Korevaar-Schoen energy,
but their approximations are not discrete and hence do not lend themselves to numerical schemes.
Neither of them work for us as we consider graphs that approximate a potentially non-graph metric
measure space, and we do not know Gromov-Hausdorff convergence of corresponding Sobolev classes. Hence we have
instead used graphs $X_r$, $r>0$, to construct a family of Dirichlet forms $\EC_r$ on $N^{1,2}(X)$, see~\eqref{eq:ECr}. We then show that 
this family of Dirichlet forms has a subsequence that $\Gamma$-converges to a Dirichlet form on $X$ so that the domain
of this limit form coincides with the Newton-Sobolev class $N^{1,2}(X)$, but with Dirichlet energy that is comparable to the
upper gradient energy therein. The definition of $\Gamma$-convergence is given in Definition~\ref{def:gamma-conv} below. 

\begin{thm}\label{thm:main1}
Let $\mu$ be a doubling measure on the complete metric space $(X,d)$ so that $(X,d,\mu)$ supports a $2$-Poincar\'e inequality. 
Then for forms $\EC_r$ defined by~\eqref{eq:ECr}, the following is true for every $u\in N^{1,2}(X)$ and its minimal 
$2$-weak upper gradient $g_u\in L^2(X)$:
\begin{enumerate}
    \item There is a constant $C>0$ such that 
    \begin{align*} 
    \frac{1}{C}\, \sup_{r>0} \EC_r[u]\le \int_Xg_u^2\, d\mu\le C\, \liminf_{\eps\to 0^+}\EC_\eps[u]. 
    \end{align*}
    \item  There is a sequence $\{r_k\}_{k\in \N}$ with $\lim_kr_k=0$, such that $\EC_{r_k}$ $\Gamma$-converges to
    a Dirichlet form $\EC$ on $X$, defined on $N^{1,2}(X)$. Moreover, for each $u\in N^{1,2}(X)$ we have that
    \[
    \frac{1}{C}\, \EC[u]\le \int_Xg_u^2\, d\mu\le C\, \EC[u].
    \]
    \item  The form $\EC$ is a positive-definite symmetric Markovian
    closed bilinear form on $N^{1,2}(X)$.
\end{enumerate}
\end{thm}

This above theorem enables us to have a way of numerically computing Dirichlet energies on a compact
doubling metric measure space supporting a $2$-Poincar\'e inequality, and in so-doing, we improve on the goal set in~\cite{DurandSh}.
In this goal, we rely on the type of Poincar\'e inequality on graphs as found in~\cite{GillLopez}. This forms the first part of the present
note. While our construction of approximating Dirichlet forms is based on the work of~\cite{GillLopez}, there are alternate
constructions based on the notion of Korevaar-Schoen energy, see~\cite{A-RB}. As the construction of~\cite{A-RB}
is not immediately related to the notion of Poincar\'e inequalities (via averages of functions on balls), and as numerical
approximations are not relevant to the construction found in~\cite{A-RB}, we do not 
consider their construction in the present paper. 

The primary reason for considering this specific construction of a Dirichlet form $\EC$ on $X$ is to have a notion of harmonic functions on
$X$ (related to this form) that is easily approximated by graph-harmonic functions. Since $X$ is compact, each approximating graph for $X$
is a finite graph, and then given boundary data for each of these graphs, it is an exercise in computational linear algebra together
with the mean value property for harmonic functions on graphs~\cite{HolSoar1,HolSoar2} to 
obtain harmonic functions on an approximating domain in the graph by considering the adjacency matrix of the domain in the graph.
Here, given the boundary data $f\in N^{1,2}(X)$ for the Dirichlet problem for $\EC$ on a domain $\Om$ in $X$, 
we may consider  a discrete mollification of $f$ as the boundary data for the Dirichlet form $\EC_r$ with respect to the
approximating domain $\Om_r\subset X$.
Observe that given a function $v\in N^{1,2}(X)$, the $\Gamma$-convergence gives us some sequence $v_k\in N^{1,2}(X)$
with $v_k\to u$ in $N^{1,2}(X)$ such that $\EC[v]=\lim_k\EC_{r_k}[v_k]$. However, in solving the Dirichlet boundary value problem,
we would need the approximating functions $v_k$ to have the same boundary values as $u$. We establish this in the second main
theorem of this paper.

\begin{thm}\label{thm:approx-zeroTrace}
Let $u\in N^{1,2}(X)$ and $\Omega$ be a domain.  
If there exist $v_j\to u$ in $N^{1,2}(X)$ and $\mathcal{E}[u] = \lim\limits_{j\to \infty} \mathcal{E}_{r_j}[v_j]$,
then there exists a sequence $u_j\in N^{1,2}(X)$ such that 
\[
\mathcal{E}[u]   =  \lim\limits_{j\to \infty} \mathcal{E}_{r_j}[u_j]
\]
and 
\[
u - u_j\in N^{1,2}_0(\Omega).
\]
\end{thm}

The rest of the paper is organized as follows. The next section, Section~\ref{Sec:2}, will provide the background information
regarding Sobolev spaces $N^{1,2}(X)$ and the related notions. Towards the beginning of this section we also provide the
description of approximating grapns $X_r$ of the metric space $X$. In Section~3 we give the construction of the approximating
forms $\EC_r$ on $X$. Sections~\ref{Sec:3} and~\ref{Sec:4} contain the proofs of Theorem~\ref{thm:main1} and
Theorem~\ref{thm:approx-zeroTrace}.

Theorem~\ref{thm:approx-zeroTrace} will be proved near the end of Section~\ref{Sec:3}. Theorem~\ref{thm:main1} will be proved in
stages. The first claim of Theorem~\ref{thm:main1} will follow from Theorem~\ref{thm:Dirich-Lowerbound}, and the existence
of the $\Gamma$-limit will be established in Proposition~\ref{prop:exist-gamma-lim}. The last part of Theorem~\ref{thm:main1},
listing the properties of the $\Gamma$-limit form $\EC$, will be established via the lemmas in Section~4. Since the approximating forms
$\EC_r$ are not known to be Markovian on $N^{1,2}(X)$, the fact that $\EC$ is Markovian does not follow directly from the theory of 
$\Gamma$-limits; instead, we establish this property in Proposition~\ref{prop:Markov}.

\section{Background}\label{Sec:2}

In this note $(X,d,\mu)$ will denote a complete geodesic
metric space $(X,d)$, equipped with a Radon measure $\mu$ supported on $X$,
such that $\mu$ is \emph{doubling}; that is, there is a constant $C_d\ge 1$ such that for all $x\in X$ and $r>0$ we have
\[
\mu(B(x,2r))\le C_d\, \mu(B(x,r)).
\]
Here, $B(x,r):=\{y\in X\, :\, d(x,y)<r\}$ is the open metric ball, and $\overline{B}(x,r):=\{y\in X\, :\, d(x,y)\le r\}$ is the closed
ball; note that the topological closure of $B(x,r)$ is contained in $\overline{B}(x,r)$, but need not equal.

\begin{defn}
With $D\subset L^2(X)$ a vector subspace,
a positive-definite symmetric bilinear form $\EC:D\times D\to\R$ is said to be a Dirichlet form on $D$ if 
\begin{enumerate}
\item $D$ is dense in $L^2(X)$ in the norm of $L^2(X)$,
\item with $\EC_1(u,v):=\EC(u,v)+\int_Xu\, v\, d\mu$ for $u, v\in D$, $(D,\EC_1)$ is a Hilbert space (in~\cite{FOT} this
property identified as the closedness property of $\EC$, see Definition~\ref{def:closed} below),
\item $\EC(u_T,u_T)\le \EC(u,u)$ whenever $u\in D$ and $u_T:=\max\{0,\min\{u,1\}\}$ (in~\cite{FOT} this
property is identified as the Markovian property).
\end{enumerate}
Associated to the Dirichlet form there is an energy form, also denoted $\EC$, with $\EC[u]:=\EC(u,u)$ for
each $u\in D$.
\end{defn}

While~\cite{FOT} sets a Dirichlet form to be $\EC:L^2(X)\times L^2(X)\to[-\infty,\infty]$ and then
sets $D$ to be the collection of all functions $u\in L^2(X)$ for which $\EC(u,u)<\infty$, this is not necessary and recent papers
on Dirichlet forms do not require $\EC$ to act on all of $L^2(X)$, see for instance~\cite[Section~5]{KuwaeShioya03},~\cite{CarronTew},
or~\cite[page~7, Section~2C]{Sturm}.

Since the doubling measure is supported on $X$, it follows that $(X,d)$ is itself a \emph{doubling metric space}; that is, 
there is some positive integer $N_d$ such that for each
$x\in X$ and $r>0$, there are at most $N_d$ number of points in $B(x,r)$ that are separated by a distance at least $r/2$ from
each other; that is, whenever $A\subset B(x,r)$ such that for $x,y\in A$ we have $d(x,y)\ge r/2$ or $x=y$, we must have that
$A$ has at most $N_d$ number of elements. It follows that for each $r>0$ we can find a set $X_r\subset X$ such that
\begin{enumerate}
\item[(a)] $d(x_1,x_2)\ge r$ whenever $x_1,x_2\in X_r$ with $x_1\ne x_2$,
\item[(b)] $X=\bigcup_{\xB\in X_r}B(\xB,r)$.
\end{enumerate}
The set $X_r$ can be given the structure of a graph, with elements of $X_r$ being the vertices of the graph and
two vertices $\xB,\yB\in X_r$ neighbors if $\xB\ne \yB$ and $B(\xB,2r)\cap B(\yB,2r)\ne\emptyset$. We denote the
neighborhood relationship by $\xB\sim\yB$. By gluing an interval of length $r$ to the two neighboring vertices $\xB,\yB$,
we obtain a metric graph, equipped with the path-metric $d_r$. Note that because $X$ is a connected metric space, necessarily
we must also have $X_r$ be a connected metric space. We refer the interested reader to~\cite{HeinonenLectures} 
for details regarding this construction.

Note that $\xB\in X_r$ now has two identities; one as a point in the graph $X_r$, and one as a point in the metric space $X$.
The following lemma compares the two metrics on $X_r$, namely the metrics $d_r$ and $d$.

\begin{lem} \label{lem_0}
Let $(X,d)$ be a doubling 
geodesic metric space and $(X_r,\dB)$ be defined as above. Then 
\[
\frac12 d(\xB,\yB)\le \dB(\xB,\yB) \le 3 d(\xB,\yB)
\]  
for all $\xB,\yB\in X_r$. 
\end{lem}

\begin{proof}
For $\xB,\yB\in X_r$ with $\xB\ne\yB$, we have that 
\[
	\dB(\xB,\yB) := Nr,
\]
where  $N$ is the smallest number for which there is a collection of points $\{\xB_i\}_{i=0}^N$ so that $\xB_0=\xB$, 
$\xB_N=\yB$ and $\xB_{i-1}\sim \xB_{i}$ for $i=1,\dots, N$. Such a collection is called an $X_r$--chain between $\xB$ and $\yB$.
Given the shortest such chain, we see that
\[
d(\xB,\yB)\leq \sum_{i=1}^N d(\xB_{i-1},\xB_{i}) \leq \sum_{i=1}^N 2r  = 2 \dB(\xB,\yB).
\]
Since $X$ is a geodesic space, we also
have a reverse relationship between $d$ and $\dB$ as well.

Let $\xB,\yB\in X_r$ be two distinct points, connected in $X$ by a geodesic $\gamma$.
If $\dB(\xB,\yB)=r$, then $r\le d(\xB,\yB)\le 4r$. If $\dB(\xB,\yB)\ge 2r$,
then $d(\xB,\yB)\ge 4r$, and we can find points $z_0:=\xB, z_1,\dots, z_{N-1},z_N:=\yB$ on $\gamma$ with the property that 
 \[
 d(z_{N_0-1},z_{N_0})\leq r, \qquad d(z_{i-1},z_{i}) = r, \text{ for } i=1,\dots, N_0-1,
 \]
and $r\le(N_0-1)r\leq d(\xB,\yB)\le N_0r$ with $N_0\ge 3$. Furthermore, each $z_i$ lies in the ball $B(\zB_i,r)$, for some $\zB_i\in X_r$
for $i=1,\cdots, N_0-1$, where we set $\zB_0=\xB, \zB_{N_0}=\yB$. Then 
$d(\zB_{i-1},\zB_{i})\leq d(\zB_{i-1},z_{i-1})+d(z_{i-1},z_{i})+d(z_{i},\zB_{i}) <3r$, and 
so $\dB(\zB_{i-1},\zB_i)\le r$ (with the value either $0$, if $\zB_{i-1}=\zB_i$, or $r$ otherwise).
Thus 
\[
\dB(\xB,\yB)\leq \sum_{i=1}^{N_0} \dB(\zB_{i-1},\zB_{i}) \le N_0r\le 3(N_0-1)r\le 3d(\xB,\yB). 
\] 
\end{proof}

The graph metric space $X_r$ plays the role of a discrete approximation
of the metric space $(X,d)$ at scale $r$. By equipping $X_r$ with a ``lift'' of the measure $\mu$, given by
\[
\mu_r(A)=\sum_{\xB\in A}\mu(B(\xB,2r))
\]
whenever $A\subset X_r$, we obtain a metric measure space $(X_r, d_r,\mu_r)$ with the measure $\mu_r$ a doubling Radon
measure on $X_r$. For more on the approximation nature of $(X_r,d_r,\mu_r)$ we refer the reader to~\cite{GillLopez}.

Given a function $u:X\to\R$, we say that a non-negative Borel measurable function $g$ on $X$ is an \emph{upper gradient} of $u$
if for every non-constant, compact, rectifiable curve $\gamma$ in $X$ we have
\[
|u(y)-u(x)|\le \int_\gamma g\, ds,
\]
where $x$ and $y$ denote the two end points of $\gamma$. If $u$ is a locally Lipschitz-continuous function on $X$, then
the local Lipschitz-constant function $\Lip u$, given by
\[
\Lip u(x)=\lim_{r\to 0^+}\, \sup_{x\ne y\in B(x,r)}\, \frac{|u(y)-u(x)|}{d(y,x)}
\]
is an upper gradient of $u$. 

For $1<p<\infty$, we say that $u\in N^{1,p}(X)$ if $u$ is measurable with $\int_X|u|^p\, d\mu<\infty$ such that $u$ has
an upper gradient $g\in L^p(X)$. We refer the reader to~\cite{HKSTbook} for more on upper gradients and 
the Sobolev class $N^{1,p}(X)$.

In this note we will also assume that $(X,d,\mu)$ supports a \emph{$p$-Poincar\'e inequality}, that is, there is a constant
$C_P\ge 1$ such that for each ball $B\subset X$ and each function-upper gradient pair $(u,g)$ on $X$ we have
\[
\jint_B|u-u_B|\, d\mu\le C_P\, \rad(B)\, \left(\jint_B g^p\, d\mu\right)^{1/p}.
\]
A more general version of Poincar\'e inequality would have the existence of a constant $\Theta\ge 1$
such that the right-hand side of the above inequality has $\Theta B$ rather than $B$. However, when $X$ is a geodesic
space (and a biLipschitz change in the metric does provide a geodesic space), it is possible to choose $\Theta=1$ at
the expense of increasing the value of $C_P$, and this is what we choose to do here. We refer the interested reader 
to~\cite{HajlaszKoskela2000, HKSTbook} for more on this.
A theorem of Cheeger~\cite{Chee} tells us that there is a linear differential structure on such $X$, commensurate with a notion
of Taylor approximation by Lipschitz charts on $X$, and a measurable inner product structure on this differential structure,
so that the inner product norm of the differential of a Lipschitz function $u$ on $X$ is comparable to $\Lip u$. However,
this proof, as given in~\cite{Chee}, is that of existence. 

In the rest of the paper we consider only the case $p=2$. The goal of this note is to give a construction of a Dirichlet form
on $X$, by using the ideas of~\cite{DurandSh, GillLopez}, together with a modification of the notion of Mosco convergence,
by considering the natural Dirichlet forms on the approximating graphs. 

\begin{remark}\label{rem:GillLopez}
It was shown in~\cite{GillLopez} that 
$(X,d,\mu)$ is doubling and supports a $2$-Poincar\'e inequality if and only if the family of approximating spaces
$(X_r,d_r,\mu_r)$ also is doubling and supports a $2$-Poincar\'e inequality, with the relevant doubling constant and the
Poincar\'e constant independent of $r$. We will make use of this result in our paper.
\end{remark}

\begin{prop}[Gill-Lopez \cite{GillLopez}]\label{prop:GL}
 Let $(X,d,\mu)$ be a doubling metric measure space  and $(X_r,\dB,\muB)$ be defined as above. 
 If $(X,d,\mu)$ admits a Poincare inequality with constants $C_{PI}>0$ and $\lambda\geq 1$, then 
 $(X_r,\dB,\muB)$ admits a Poincare inequality in the Holopainen-Soardi sense: there are constants 
 $C_{HS}>0$ and $\theta\geq 1$ independent of $r$ such that for any $\xB_0\in X_r$, $R\geq 0$ and 
 $\uB$ defined on $X_r$, we have  
 \[
 \sum_{\xB\in G_R} |\uB(\xB) - (\uB)_{G_R}|  \frac{\muB(\xB)}{\muB(G_R)} \leq C_{HS} \cdot R 
 \left [\sum_{\xB\in G_{\theta R}} \sum_{\yB\sim \xB} \left |\frac{\uB(\xB) - \uB(\yB) }{r}\right |^p  
 \frac{\muB(\xB)}{\muB(G_{\theta R})}\right ]^{1/p},
 \] 
 where  $G_R=G(\xB_0;R) = \{\xB\in X_r\, :\, \dB(\xB, \xB_0)\leq R  \}$.
 \end{prop}

Given the graph $X_r$ as above, the graph energy of a function $u_r:X_r\to\R$ is given by
\begin{equation}\label{eq:Graph-energy}
\sum_{\xB\in X_r}\, \sum_{\yB\sim\xB}\frac{|u_r(\yB)-u_r(\xB)|^2}{r^2}\, \mu_r(\xB).
\end{equation}

We straddle three realms here; the graph energy on the graph $X_r$, the energy $\EC_r$ induced on $X$
by the graph energy (see~\eqref{eq:ECr} below), and the $\Gamma$-limit energy $\EC$ on $X$. The graph energy
is directly seen to be derived from a Dirichlet form, and we will see in the last section that the $\Gamma$-limit energy
$\EC$ is also derived from a Dirichlet form. The energy $\EC_r$, however, is not known to be from a Dirichlet form, though it
is derived from a positive-definite symmetric bilinear form.

\section{Constructing approximating  energy forms $\EC_r$}\label{Sec:3}

In this section we use the structure of $X_r$ to construct an approximating Dirichlet form 
$\EC_r:N^{1,2}(X)\times N^{1,2}(X)\to [-\infty,\infty]$ on $X$ as follows. For functions $u\in N^{1,2}(X)$ and $\xB\in X_r$, we set $u_r:X_r\to\R$ by
\begin{equation}\label{eq:ur-construction}
u_r(\xB):=\jint_{B(\xB,r/4)}u\, d\mu,
\end{equation}
and set the Dirichlet form by the formula
\begin{equation}\label{eq:ECr}
\EC_r(u,v):=\sum_{\xB\in X_r}\sum_{\yB\sim \xB}\, \frac{[u_r(\yB)-u_r(\xB)][v_r(\xB)-v_r(\yB)]}{r^2}\, \mu_r(\xB)
\end{equation}
whenever $u,v\in N^{1,2}(X)$. It is clear that $\EC_r$ is a bilinear form on the Hilbert space $N^{1,2}(X)$.
We set $\EC_r[u]:=\EC_r(u,u)$.

\begin{remark}
The form $\EC_r$ is a bilinear energy form but is not known to satisfy the Markov property, and so should, according
to the definition of~\cite{BD1, FOT}, not be called a Dirichlet form. 
Moreover, $\EC_r$ need not be closed as it acts only on $N^{1,2}(X)$, but it is closable in $L^2(X)$. The point here
is that we can have a sequence $(u_n)_n$ of functions in $N^{1,2}(X)$ that converges to a function in $L^2(X)\setminus N^{1,2}(X)$
and so that $\EC_r(u_n-u_m,u_n-u_m)\to 0$ as $n,m\to\infty$.
However, $\EC_r$ is induced by a bilinear energy form
on the graph $X_r$ which is indeed a Dirichlet form, and so we allow ourselves the liberty of calling $\EC_r$ also a Dirichlet
form as this does not give rise to any confusion. The $\Gamma$-limit energy form on $N^{1,2}(X)$ obtained from
the forms $\EC_r$ will indeed be a Dirichlet form, see Proposition~\ref{prop:Markov} and Lemma~\ref{lem:closed} below.
\end{remark}

We are also concerned with the natural energy form on $X$ inherited from the upper gradient structure; for $u\in N^{1,p}(X)$
we set 
\[
E[u]:=\inf_g\int_X g^2\, d\mu, 
\]
where the infimum is over all upper gradients $g$ of $u$. From the results in~\cite{HKSTbook} we know that there is a minimal
$2$-weak upper gradient, $g_u$, associated with each $u\in N^{1,2}(X)$, such that 
\[
E[u]=\int_X g_u^2\, d\mu.
\]
When $u$ is locally Lipschitz on $X$, it follows from~\cite{Chee} together with~\cite{HKSTbook} that $g_u=\Lip u$.
For functions $u\in L^2(X)$ that do not have (a modification on a set of measure zero) representative in $N^{1,2}(X)$ we set
$E[u]:=\infty$.

The following proposition gives us a comparison between the family $\EC_r$ and $E$.

\begin{prop}\label{prop:Dirich-upper}
Let $X$ be a doubling metric measure space admitting a $2$-Poincar\'e inequality. Then there exists $C>0$ such that 
$\EC_r[u]\le C E[u]$
for all $r>0$ and $u\in N^{1,2}(X)$. Moreover, whenever  $\Lip\, u$ is continuous and belongs to $L^2(X,d\mu)$ ,we have 
\[
E[u]\leq C \liminf_{r\to 0+} \EC_r[u].
\]
The constant of comparison, $C$, is dependent
solely on the doubling constant $C_d$ associated with $\mu$, the metric doubling constant $N_d$ associated with the doubling
metric space $(X,d)$, and the scaling constant $\theta$ from Proposition~\ref{prop:GL}; in particular, it is independent of $u$ 
and $r$.
\end{prop}

This proposition will be proved using Lemma~\ref{lem_4} below.

\begin{lem}\label{lem_4}
Let $u\in N^{1,2}(X)$ be a locally Lipschitz function such that $\Lip u$ is a continuous function on $X$. 
Then for any $r_0>0$ and any $x\in X$ 
there exists a ball $B_x$, centered at $x$ and 
of radius $r(B_x)\leq r_0$, such that for all $\eps\leq \rad(B_x)/36$ we have
\begin{equation} \label{eq_lem1}
\Lip u(x)\lesssim
\sum_{\xB,\yB\in E(x,r(B_x))} \frac{|u_\eps(\xB)-u_\eps(\yB)|}{r(B_x)} \frac{\mu_\eps(\xB) \mu_\eps(\yB)}{\mu(B_x)^2},
\end{equation}
where we set 
\[
E(x,r(B_x))=\{\xB\in X_\eps: B(\xB;\eps)\subset B_x\}.
\] 
\end{lem}

\begin{proof}
We fix $r_0>0$ and $x\in X$. Without loss of generality we may assume that $\Lip u(x)>0$. Then by the continuity of
$\Lip u$, there is some positive number $r_1\leq r_0$ such that 
\begin{equation} \label{lem3_eq0}
\sup_{y\in B(x,r_1)} |\Lip u(y) - \Lip u(x)| \leq \frac{1}{2}\Lip u(x). 
\end{equation}  
Next, by the definition of $\Lip u(x)$, there exists a positive number $r_2\leq r_1/2 $ and 
a point $y\in B(x,r_2)$ such that 
\begin{equation}\label{eq:Lip-at-x-1}
\sup_{x'\in B(x,r_2)} \frac{|u(x') - u(x)|}{d(x,x')} \leq \frac32  \Lip u(x) 
\end{equation}
and
\begin{equation}\label{eq:Lip-at-x}
\frac{|u(y) - u(x)|}{r_2} \geq \frac{2}{3}  \Lip u(x).
\end{equation}
To see this claim, we argue as follows. 
By the definition of $\Lip u(x)$ we know that there is some $y\in B(x,r_1/2)$ such that $y\ne x$ and 
\[
\frac{|u(y)-u(x)|}{d(y,x)}\ge \frac23 \Lip u(x).
\]
Setting $r_2=d(x,y)$, we obtain~\eqref{eq:Lip-at-x}. By~\eqref{lem3_eq0} we know that
\[
\sup_{w\in B(x,r_1)}\, \Lip u(w)\le \frac32 \Lip u(x),
\]
and as for each $x'\in B(x,r_2)$ we know that $B(x',r_2)\subset B(x,r_1)$, we know that $u$ is $\frac32 \Lip u(x)$-Lipschitz
continuous on $B(x,r_2)$, and so~\eqref{eq:Lip-at-x-1} follows.

We set $B_x:=B(x,r_2)$ and will show that \eqref{eq_lem1} holds for $\eps\leq r_2/36$. 
In order to do this,  consider the balls $B_1=B(x,r_2/9)$ and $B_2=B(y,r_2/9)$. Then by the choice of 
$r_2$ and by~\eqref{eq:Lip-at-x-1}, for all $x'\in B_1$ 
we have
\[
|u(x')-u(x)|\le \frac32\, \Lip u(x)\, d(x,x')<\frac{r_2}{6}\, \Lip u(x).
\]
Therefore
\begin{equation*}
\frac{|u(x') - u(x)|}{r_2} \leq 
	\frac{1}{6} \Lip u(x).
\end{equation*}
Similarly, because $X$ is a geodesic space and $\Lip u$ is an upper gradient of the Lipschitz function $u$ 
(see~\cite{HeinonenLectures}, \cite{HKSTbook}) and by~\eqref{lem3_eq0} we see that
for all $y'\in B_2$,
\[
|u(y')-u(y)|\le \sup_{z\in B(y,r_2)}\Lip u(z)\ d(y,y')\le \frac32\, \Lip u(x)\  \frac{r_2}{9}=\frac{r_2}{6}\, \Lip u(x).
\]	
Note that $d(x,y)=r_2$, and so 
\[
E(x,r_2/9):= \left \{\xB\in X_\eps: B(\xB,\eps)\subset B_1\right\} \subset E(x,r_2),
\] 
\[
E(y,r_2/9) := \left \{\xB\in X_\eps: B(\xB,\eps)\subset B_2\right \} \subset E(x,r_2).
\]
For $\xB\in E(x,r_2/9)$ and $\yB\in E(y,r_2/9)$	we have 
\[
\frac{|u_\eps(\xB)-u(x)|}{r_2} \leq \fint_{B(\xB,\eps/4)} \frac{|u(x')-u(x)|}{r_2}\, d\mu(x') \leq \frac{1}{6} \Lip u(x)
\]
and 
\[
\frac{|u_\eps(\yB)-u(y)|}{r_2} \leq \fint_{B(\yB,\eps/4)} \frac{|u(y')-u(y)|}{r_2} d\mu(y') \leq \frac{1}{6} \Lip u(x).
\] 
Hence by~\eqref{eq:Lip-at-x}, for all $\xB\in E(x,r_2/9)$ and $\yB\in E(y,r_2/9)$,
\begin{equation}\label{eq:Lip-bd}
\frac{|u_\eps(\xB) - u_\eps(\yB) |}{r_2}\geq \frac{1}{3} \Lip u(x).
\end{equation}
		 
As the balls $B(\xB,\eps/2)$ are disjoint in $X$ and the measure $\mu$ is doubling, we have
\begin{align*}
\mu_\eps(E(x,r_2/9))=\sum_{\xB \in E(x,r_2/9)} \mu_\eps(\xB) 
		 &= \sum_{\xB \in E(x,r_2/9)} \mu(B(\xB,\eps))\\
		 & \lesssim \sum_{\xB \in E(x,r_2/9)} \mu(B(\xB;\eps/2)) \\
		 &\leq \mu(B_1).
\end{align*}
Similarly we obtain
\[
\mu_\eps(E(y,r_2/9))\lesssim \mu(B_2).
\]
On the other hand, as $\eps\leq r_2/36$ and $X_\eps$ is an $\eps$-net, we have that
\[
\tfrac12 B_1\subset \bigcup_{\xB\in E(x,r_2/9)}B(\xB;\epsilon)\ \text{ and }\
\tfrac12 B_2\subset \bigcup_{\yB\in E(y,r_2/9)}B(\yB;\epsilon).
\]

Therefore 
\begin{align*}
\mu(B_1)&\lesssim\mu(\tfrac12 B_1)\le \sum_{\xB\in E(x,r_2/9)}\mu(B(\xB;\eps))=\mu_\eps(E(x,r_2/9)), \\
\mu(B_2)&\lesssim\mu(\tfrac12 B_2)\le \sum_{\yB\in E(y,r_2/9)}\mu(B(\yB;\eps))=\mu_\eps(E(y,r_2/9)).
\end{align*}
Combining the above two sets of comparisons, we obtain
\[
\mu_\eps(E(x,r_2/9))\approx \mu(B_1)\  \text{ and }\  \mu_\eps(E(y,r_2/9))\approx\mu(B_2),
\]
and as $r(B_1)=r(B_2)=r(B_x)/9$ with the centers $x$ and $y$, of $B_1$ and $B_2$ respectively, in $B_x$, we obtain
from the doubling property of $\mu$ that 
\[
\mu_\eps(E(x,r_2/9))\approx \mu(B_x) \ \text{ and }\ \mu_\eps(E(y,r_2/9)) \approx \mu(B_x). 	 
\]
		 Finally, by~\eqref{eq:Lip-bd} we obtain 
\begin{align*}
\Lip u(x)&\leq 3 \sum_{\xB\in E(x,r_2/9)} 
\sum_{\yB\in E(y,r_2/9)} \frac{|u_\eps(\xB) - u_\eps(\yB) |}{r_2} \frac{\mu_\eps(\xB)}{\mu_\eps(E(x,r_2/9))} 
  \frac{\mu_\eps(\yB)}{\mu_\eps(E(y,r_2/9))}  \\
& \leq 3 \sum_{\xB,\yB\in E(x,r(B_x))}  \frac{|u_\eps(\xB) - u_\eps(\yB) |}{r_2} 
   \frac{\mu_\eps(\xB)}{\mu_\eps(E(x,r_2/9))} \frac{\mu_\eps(\yB)}{\mu_\eps(E(y,r_2/9))} \\
& \lesssim\sum_{\xB,\yB\in E(x,r(B_x))}  \frac{|u_\eps(\xB) - u_\eps(\yB) |}{r_2} 
   \frac{\mu_\eps(\xB)\, \mu_\eps(\yB)}{\mu(B_x)^2}. 
\end{align*}
\end{proof}

Now we are ready to prove Proposition~\ref{prop:Dirich-upper}.

\begin{proof}[Proof of Proposition~\ref{prop:Dirich-upper}]
Let $u\in N^{1,2}(X)$ be a locally Lipschitz continuous function on $X$ such that its local Lipschitz constant
function $\Lip u$ is also continuous on $X$.
Without loss of generality, we assume that $\int_X (\Lip u)^2\, d\mu$ is positive, for otherwise $u$ is constant and all
the relevant Dirichlet energies are zero. 
Now we introduce a large parameter $\lambda>1$ that we will fix at the end (indeed, it can be 
any choice of $\lambda>12\theta+1$ with  
$\theta$ from the Poincar\'e inequality given in Proposition~\ref{prop:GL} above).
Let $K$ be a closed ball of sufficiently large radius so that 
\[
\int_X (\Lip u)^2 \ d\mu \leq 2 \int_K (\Lip u)^2 \  d\mu.  
\]  
Note that as $X$ is complete and doubling, it is proper; hence $K$ is compact.
By uniform continuity of $\Lip u$ on $10\lambda K$, we can find $r_1\in (0,1)$ small enough so that 
$2r_1$ is less than the radius of $K$ and 
\begin{equation}\label{eq:Lip-unif-cont}
	\sup_{\substack{x,y\in 10\lambda K\\ d(x,y)<2r_1}} |(\Lip u(x))^2-(\Lip u(y))^2| 
\leq \frac{1}{\Lambda}\fint_K (\Lip u)^2 \ d\mu.
\end{equation}
Here we have fixed a number $\Lambda>2$, to be chosen later.

For each $x\in K$ we find a ball $B_x$ satisfying the conclusions of Lemma~\ref{lem_4} 
with the choice of $r_0=r_1/(5\lambda)$. Such a family of balls is a cover of the compact set $K$, and so we can
obtain a a finite subcover of $K$. Moreover, by the $5r$-covering lemma 
(see e.g.~\cite[Theorem~1.2]{HeinonenLectures} or~\cite[page~60]{HKSTbook}), we can further find a 
subcollection of balls $\{B_i\}_{i=1}^N$ such that 
\begin{equation}
\{\lambda B_i\} \text{ are disjoint and } \bigcup_{i=1}^N 5\lambda B_i \supset K.
\end{equation}
With this subcover, we get 
\[
\int_K (\Lip u)^2 \ d\mu \leq \sum_{i=1}^N \left(\fint_{5 \lambda B_i} (\Lip u)^2 \ d\mu\right) \cdot \mu(5 \lambda B_i),
\]
where the balls $B_i$ satisfy the conclusions of Lemma~\ref{lem_4}. We denote the centers of these balls by $x_i$. 
Since the radii of each $5\lambda B_i$ is no more than $r_1$ and the center of $5\lambda B_i$ lies in $K$, 
it follows from~\eqref{eq:Lip-unif-cont} that
\[
\fint_{5 \lambda B_i} (\Lip u)^2 \ d\mu\le (\Lip u(x_i))^2+\frac{1}{\Lambda}\, \fint_K(\Lip u)^2\, d\mu.
\] 
We now obtain from the above observations that
\begin{align*}
\int_K (\Lip u)^2\, d\mu &\le \sum_{i=1}^N (\Lip u(x_i))^2 \cdot \mu(5 \lambda B_i)
   +\left(\fint_K(\Lip u)^2\, d\mu\right)\, \sum_{i=1}^N\mu(5 \lambda B_i)\\
   &\lesssim \sum_{i=1}^N (\Lip u(x_i))^2 \cdot \mu(5 \lambda B_i)+\frac{1}{\Lambda}\, \fint_K(\Lip u)^2\, d\mu\sum_{i=1}^N\mu(B_i)\\
   &\le \sum_{i=1}^N (\Lip u(x_i))^2 \cdot \mu(5 \lambda B_i)+\frac{\mu(2K)}{\Lambda}\, \fint_K(\Lip u)^2\, d\mu\\
   &\le \sum_{i=1}^N (\Lip u(x_i))^2 \cdot \mu(5 \lambda B_i)+\frac{1}{\Lambda}\, \int_K(\Lip u)^2\, d\mu.
\end{align*}
The last few steps in the above argument was arrived at by using the fact that the balls $B_i$ are pairwise disjoint and are
contained in $2K$. The comparison constant associated with the above series of inequalities, $C_*$, depends solely on 
the doubling constant of $\mu$ and the choice of $\lambda$.
We choose $\Lambda\ge 2C_*$. It follows that $C_*/\Lambda\le 1/2$, and so
\[
\int_K(\Lip u)^2\, d\mu\le 2\, C_*\, \sum_{i=1}^N (\Lip u(x_i))^2 \cdot \mu(5 \lambda B_i).
\]
 	
Recall that in Lemma~\ref{lem_4}, the parameter $\eps_0$ depended on the radius of the ball $B_x$.
We will indicate this dependence by denoting $\eps_0$ by $\eps_0(x)$.
Therefore, choosing any $\eps\leq \min_{1\leq i\leq N} \eps_0(x_i)$, we get 
\begin{align*}
\int_K (\Lip u)^2 \ d\mu &\lesssim \sum_{i=1}^N  
\left[\sum_{\xB, \yB \in {E(x_i,r(B_{i}))}} \frac{|u_\eps(\xB)-u_\eps(\yB)|}{r(B_i)} \frac{\mu_\eps(\xB) \mu_\eps(\yB)}{\mu(B_i)^2 } \right]^2 \mu(B_i)\\
&\lesssim \sum_{i=1}^N\left[\sum_{\xB\in E(x_i,r(B_i))}\frac{|u_\eps(\xB)-(u_{\eps})_{G(x_i,r(B_i))}|}{r(B_i)} \frac{\mu_\eps(\xB)}{\mu(B_i)}\right]^2\mu(B_i),
\end{align*}
where, as in Lemma~\ref{lem_4}, we have
\[
\sum_{\yB\in E(x_i,r(B_i))}\mu_\eps(\yB)\approx \mu(B_i)
\]
and
\[
E(x_i,r(B_i))=\{\xB\in X_\eps\, :\, B(\xB;\eps)\subset B_i\}\subset G(x_i,r(B_i))=\{\xB\in X_\eps\, :\, d_\eps(x_i,\xB)\le r(B_i)\}.
\]
By the doubling property of $\mu$ again, we have that $\mu_\eps(G(x_i,r(B_i)))\approx\mu(B_i)$.
Now we apply Proposition~\ref{prop:GL}, to obtain
\begin{align*}
\int_K (\Lip u)^2 \ d\mu &\lesssim \sum_{i=1}^N\left[\sum_{\xB\in G(x_i,\theta r(B_i))}\, 
\sum_{\yB\sim\xB}\frac{|u_\eps(\xB)-u_\eps(\yB)|^2}{\eps^2}\, \frac{\mu_\eps(\xB)}{\mu(B_i)}\right]\mu(B_i)\\
&\lesssim \sum_{i=1}^N\left[\sum_{\xB\in G(x_i,\theta r(B_i))}\, 
\sum_{\yB\sim\xB}\frac{|u_\eps(\xB)-u_\eps(\yB)|^2}{\eps^2}\, \mu_\eps(\xB)\right]\\
&\lesssim \sum_{\xB\in X_\eps}\sum_{\yB\sim\xB}\, \left[\sum_{i=1}^N\chi_{G(x_i,\theta r(B_i))}(\xB)\right]\, 
\frac{|u_\eps(\xB)-u_\eps(\yB)|^2}{\eps^2}\, \mu_\eps(\xB),
\end{align*}
where we used Tonelli's theorem to obtain the last inequality. Since $\lambda>12\theta+1$ and the balls $\lambda B_i$, $i=1,\cdots, N$ are
pairwise disjoint, and as for each $i$ we have that $G(x_i,\theta r(B_i))\subset 2\theta B_i\subset \lambda B_i$, we have that
\[
\int_K (\Lip u)^2 \  d\mu
\lesssim \sum_{\xB\in X_\eps}\sum_{\yB\sim\xB}\,  
\frac{|u_\eps(\xB)-u_\eps(\yB)|^2}{\eps^2}\, \mu_\eps(\xB)=\mathcal{E}_\eps[u].
\]
The above holds for all sufficiently small $\eps>0$, and so we have that
\[
\int_Xg_u^2\, d\mu=\int_X(\Lip u)^2\, d\mu\le 2\, \int_K(\Lip u)^2\, d\mu\lesssim \liminf_{\eps\to 0^+}\mathcal{E}_\eps[u],
\]
which verifies the claimed latter inequality in the statement of the proposition.

We now complete the proof of the proposition by verifying the first inequality stated in the claim.
To this end, let $u\in N^{1,2}(X)$. We fix $r>0$ such that $r<\text{diam}(X)/2$.
Recall that for distinct $\xB,\yB\in X_r$ with $\xB\sim \yB$ we have $d(\xB,\yB)\in [r,4r)$. Thus
\[
B(\yB;r) \subset B(\xB;5r)\subset B(\yB;9r).
\]
As $\mu$ is doubling, it follows that 
\[
\frac{\mu(B(\xB,5r))}{\mu(B(\yB,r))}\le \frac{\mu(B(\yB,9r))}{\mu(B(\yB,r))}\le C_d^4\ \text{ and }\ 
\frac{\mu(B(\xB,5r))}{\mu(B(\xB,r))}\le C_d^3.
\]
Hence, if $\xB\sim \yB$, then 
\begin{align*}
|u_r(\xB) - u_r(\yB)|  &\leq  \fint_{B(\xB;r/4)} \fint_{B(\yB;r/4)} |u(x) - u(y)| d\mu(y) d\mu(x) \\ 
&  \lesssim \fint_{B(\xB;5r)} \fint_{B(\xB;5r)} |u(x) - u(y)| d\mu(y) d\mu(x) \\ 
& \lesssim \fint_{B(\xB;5r)} |u(x) - u_{B(\xB;5r)}|  d\mu(x).
\end{align*}
So 
\begin{align*}
\EC_r[u]  &= \sum_{\xB\in X_r} \sum_{\xB\sim \yB} \left | \frac{u_r(\xB) - u_r(\yB)}{r} \right |^2 \muB(\xB) \\  
& \lesssim N_d\cdot  \sum_{\xB\in X_r} \left |(5r)^{-1} \fint_{B(\xB;5r)} |u(x) - u_{B(\xB;5r)}|  d\mu(x) \right |^2  \muB(\xB),   
\end{align*}
where $N_d$ is the doubling constant of the metric space $(X,d)$.   
Next, applying the Poincar\'e inequality, we get 
\begin{align*}
\EC_r[u]  \lesssim \sum_{\xB\in X_r} \fint_{B(\xB;5r)} (\Lip u)^2 \  d\mu \cdot \mu(B(\xB;r)) 
&\lesssim \sum_{\xB\in X_r} \int_{B(\xB;5 r)} (\Lip u)^2 \  d\mu \\
& \leq   N_{d}^3\,  \int_{X} (\Lip u)^2 \  d\mu, 
\end{align*}
where we have used the fact that two balls $B(\xB,5r)$ and $B({\xB}',5r)$ intersect with $\xB\ne {\xB}'$ if and only if
$r\le d(\xB,{\xB}')<5r$. 
The constants of comparison in the above inequalities all depend only on the doubling constant $C_d$ of 
the measure $\mu$, and in particular are independent of $u$ and $r$.
\end{proof}

\begin{thm}\label{thm:Dirich-Lowerbound}
Let $X$ be a doubling metric measure space admitting a $2$-Poincar\'e inequality. 
Then there is a constant $C>0$ such that for each $u\in N^{1,2}(X)$
we have
\begin{align} \label{ineq.Dirich-lower-bound}
\frac{1}{C}\, \sup_{r>0} \EC_r[u]\le E[u]=\int_Xg_u^2\, d\mu\le C\, \liminf_{\eps\to 0^+}\EC_\eps[u].
\end{align}
Moreover, if $u_k$, $k\in\N$, is a sequence of functions in $N^{1,2}(X)$ with $u_k\to u$ in $N^{1,2}(X)$, then
$\lim_k\EC_r[u_k]=\EC_r[u]$ for each $r>0$.
\end{thm}

\begin{proof}
The first inequality follows immediately from Proposition~\ref{prop:Dirich-upper}, and so only the second inequality needs verification.
That is the focus of this proof.

Let $u\in N^{1,2}(X)$, and $(u_k)_k$ be a sequence of compactly supported Lipschitz functions on $X$ with compactly
supported Lipschitz upper gradients, such that $u_k\to u$ in $N^{1,p}(X)$; see~\cite[Theorem 5.3]{Chee} (note that in~\cite{Chee},
the continuous upper gradient $\hat{v}_{k,j(k)}$ used to construct the approximating function $u_k=\hat{f}_{k,j(k)}$,
by the construction of $u_k$, necessarily has the property that $\Lip u_k=\hat{v}_{k,j(k)}$). 
By passing to a subsequence if necessary, we may assume that  
$\Vert u-u_k\Vert_{N^{1,2}(X)}\le 1/k$.
Then by the first part of Proposition~\ref{prop:Dirich-upper} again, we have that
\[
\EC_r[u-u_k]\le C\, \int_Xg_{u-u_k}^2\, d\mu\le \frac{C}{k^2}.
\]
Note that the constant $C$ independent of $r$.
Therefore, for each $\eps>0$ we can find a positive integer $k_\eps$ such that whenever $k_\eps\le k\in\N$,
for all $r>0$ we have $1/k<\eps$ and 
\[
\EC_r[u-u_k]\le \eps^2.
\]
Now by the last inequality in the statement of Proposition~\ref{prop:Dirich-upper}, there is some
$r_\eps>0$ such that for each $0<r<r_\eps$, we have 
\[
\int_Xg_{u_k}^2\, d\mu\le C\, \EC_r[u_k].
\]
For such $r$ and $k$, we have by triangle inequality,
\begin{align*}
\EC_r[u]^{1/2}\ge \EC_r[u_k]^{1/2}-\EC_r[u-u_k]^{1/2} 
 \ge \frac{1}{C}\left(\int_Xg_{u_k}^2\, d\mu\right)^{1/2}-\eps. 
\end{align*}
The desired conclusion follows from letting $r\to 0^+$ and then $\eps\to 0$.
 
We now show that if $(u_k)_k$ is a sequence
in $N^{1,2}(X)$ such that $u_k\to u$ in $N^{1,2}(X)$, then for each $r>0$ we have that $\EC_r[u_k-u]\to 0$ as $k\to\infty$.
We fix $r>0$ and consider the balls $B_j$, $j\in\N$, related to the construction of $X_r$; namely the balls of the 
form $B(\xB,r)$, $\xB\in X_r$. Note that $\xB\sim \yB$ if and only if $B(\xB;2r)\cap B(\yB;2r)$ is nonempty.
Now setting $v_k=u_k-u$, we see that 
\begin{align*}
\EC_r[u_k-u]=\EC_r[v_k] 
&=\sum_{\xB\in X_r}\, \sum_{\yB\in X_r\, :\, \xB\sim \yB}\frac{|(v_k)_{B(\xB;r/4)}-(v_k)_{B(\yB;r/4)}|^2}{r^2}\mu(B(\xB;r))\\
&\le \frac{4}{r^2}\sum_{\xB\in X_r}\, \sum_{\yB\in X_r\, :\, \xB\sim \yB}\left((|v_k|)_{B(\xB;r/4)}^2+(|v_k|)_{B(\yB;r/4)}^2\right)\mu(B(\xB;r))\\
&\le \frac{4}{r^2}\sum_{\xB\in X_r}\, \sum_{\yB\in X_r\, :\, \xB\sim \yB}\left(\fint_{B(\xB;r/4)}|v_k|^2\, d\mu
  +\fint_{B(\yB;r/4)}|v_k|^2\, d\mu\right)\mu(B(\xB;r)).
\end{align*}
Note that if $\xB\sim \yB$, then $B(\yB;r)\subset B(\xB;5r)$; moreover, there is at most $N$ such $\yB\in X_r$ that are
neighbors of $\xB$, with $N$ independent of $\xB$ and $r$ (in fact, it depends solely on the metric doubling constant $N_d$).
Therefore by the doubling property of $\mu$ we have
\[
\EC_r[u_k-u]=\EC_r[v_k] 
\lesssim \frac{1}{r^2}\sum_{\xB\in X_r} \int_{B(\xB; 5r)}|v_k|^2\, d\mu\lesssim \frac{1}{r^2}\int_X|v_k|^2\, d\mu,
\]
where in the last step we have used the bounded overlap property $\sum_{\xB\in X_r}\chi_{B(\xB; 5r)}\le K$ on $X$
with $K$ depending solely on the doubling constant $N_d$. Thus, if 
$u_k\to u$ in $L^2(X)$, that is, $v_k\to 0$ in $L^2(X)$, it follows that $\EC_r[u_k-u]\to 0$ as well as $k\to\infty$.
Note that 
\[
\EC_r[u]^{1/2}-\EC_r[u_k-u]^{1/2}\le \EC_r[u_k]^{1/2}\le \EC_r[u]^{1/2}+\EC_r[u_k-u]^{1/2}.
\]
This completes the proof of the theorem; though note that the convergence of $\EC_r[u_k-u]$ to zero is not
uniform across $r>0$ as we have $r^2$ in the denominator of the above estimate.
\end{proof}

\begin{defn}\label{def:gamma-conv}
A sequence $\EC_{r_k}$ of Dirichlet forms on the Hilbert space $N^{1,2}(X)$ 
is said to $\Gamma$-converge to a Dirichlet form $\EC$ if
the following two conditions are satisfied:
\begin{enumerate}
\item[(i)] Whenever $(u_k)_k$ is a sequence in $N^{1,2}(X)$ and $u\in N^{1,2}(X)$ such that $u_k\to u$ in $N^{1,2}(X)$, 
we must have $\EC[u]\le \liminf_k \EC_{r_k}[u_k]$;
\item[(ii)] For each $u\in N^{1,2}(X)$ there is a sequence $(u_k)_k$ in $N^{1,2}(X)$ such that $u_k\to u$ in $N^{1,2}(X)$ and
$\EC[u]\ge \limsup_k\EC_{r_k}[u_k]$.
\end{enumerate}
Given the first condition above, the second condition is equivalent to knowing the existence of a sequence $(u_k)_k$ in 
$N^{1,2}(X)$ such that $\EC[u]=\lim_k\EC_{r_k}[u_k]$ and $u_k\to u$ in $N^{1,2}(X)$.
\end{defn}

\begin{prop}\label{prop:exist-gamma-lim}
There exists a $\Gamma$-convergent subsequence $\EC_{r_k}$ of $\EC_r$.
\end{prop}

\begin{proof}
From~\cite{AlvarHajMal} we know that $N^{1,2}(X)$ is separable. 
Therefore, by~\cite[Theorem 8.5]{DalMaso}, 
there exists a $\Gamma$-convergent subsequence $\EC_{r_k}$ of $\EC_r$. 
From~\cite[Proposition~6.8]{DalMaso} we know that the $\Gamma$-limit is a closable bilinear energy form.
\end{proof}

We denote the $\Gamma$-limit of $\EC_{r_k}$ from the previous proposition by $\EC$. We now show that the 
energy form induced by the $\Gamma$-limit $\EC$ is comparable to the Newton-Sobolev energy form, thus proving
Condition~(2) of Theorem~\ref{thm:main1}.

\begin{prop} \label{prop:Comparability}
There is a constant $C>0$ such that the following is true for every $u\in N^{1,2}(X)$ and its minimal 
$2$-weak upper gradient $g_u\in L^2(X)$:
\begin{align} \label{eq:Comparability}
\frac{1}{C}\, \EC[u]\le \int_Xg_u^2\, d\mu\le C\, \EC[u].
\end{align}
\end{prop}

\begin{proof}
If $u_k$, $k\in\N$, is the sequence guaranteed by the $\Gamma$-limit so that
$\lim_k\EC_{r_k}[u_k]=\EC[u]$ and $\Vert u_k-u\Vert_{N^{1,2}(X)}\to 0$ as $k\to\infty$, then we have that $\int_Xg_u^2\, d\mu=\lim_k\int_Xg_{u_k}^2\, d\mu$. By~\cite[Proposition~6.8]{DalMaso}, we know that the $\Gamma$-limit is lower semicontinuous. Then by Theorem~\ref{thm:Dirich-Lowerbound} we have that
\begin{align*}
\EC[u] \le \liminf_k \EC[u_k] \le C\liminf_k\int_Xg_{u_k}^2\, d\mu=C\int_Xg_u^2\, d\mu,
\end{align*}
and so $\EC[u]\le C\, \int_Xg_u^2\, d\mu$, which gives the left side of \eqref{eq:Comparability}. On the other hand, by Theorem~\ref{thm:Dirich-Lowerbound}, we have
\[
\EC[u]^{1/2}=\lim_k\EC_{r_k}[u_k]^{1/2}\ge \liminf_k\EC_{r_k}[u]^{1/2}-\limsup_k\EC_{r_k}[u-u_k]^{1/2}
\ge \left(\frac{1}{C}\int_Xg_u^2\, d\mu\right)^{1/2},
\]
where we have used the fact that (Again from Theorem~\ref{thm:Dirich-Lowerbound}),
\begin{equation}\label{eq:not-bright}
0\le\limsup_k\EC_{r_k}[u-u_k]\le \limsup_k\int_Xg_{u_k-u}^2\, d\mu=0.
\end{equation}
This proves Claim~(2) of Theorem~\ref{thm:main1}.
\end{proof}

One of the key properties needed with respect to the $\Gamma$-convergence is that if the function being approximated 
vanishes outside a bounded domain $\Om\subset X$ with $\mu(X\setminus\Om)>0$, then the sequence approximating
this function in the $\Gamma$-convergence condition~(ii) also vanishes outside the approximating domain. This property
is essential in the discussion on approximating solutions to the Dirichlet problem on $\Om$, and is called the 
\emph{matching boundary values} in~\cite[Section~4.2.1]{Braides}. 
However, the method of De Giorgi, as outlined in~\cite{Braides}, uses a decomposition of the domain $\Om$ into suitable
annular rings and constructing a partition of unity subordinate to this decomposition. 
An alternate method is to construct different $\Gamma$-limits for different boundary data as in~\cite[Chapter~21]{DalMaso},
but this again is computationally involved (see~\cite[Theorem~21.1]{DalMaso}). 
In the second main theorem, Theorem~\ref{thm:approx-zeroTrace}, 
we give a more direct proof of the existence of optimal sequences with matching boundary
values in the sense of~\cite{Braides}. We now prove this theorem.

\begin{proof}[proof of Theorem~\ref{thm:approx-zeroTrace}]
For $x\in X$ set $w(x) = \dist(x,X\setminus \Omega)$. For $j\in\N$ let 
    \[
        u_j = \min\{\max\{u-w,v_j\}, u+w\}. 
    \]
 Then, 
    \[
        u_j-u = \min\{\max\{-w,v_j-u\}, w\},
    \]
and therefore $u_j-u \in N^{1,2}_0(\Omega)$. 
Further, note that as $w\geq 0$, when $x\in X$ we have that $u_j(x)-u(x)>0$ if and only if $v_j(x)-u(x)>0$, and in 
which case we also have $u_j-u(x)\leq v_j-u(x)$. If $u_j(x)-u(x)\le 0$, then necessarily $v_j(x)-u(x)\le 0$ 
and $u_j-u(x)\geq v_j-u(x)$. Thus, $|u_j(x)-u(x)|\leq |v_j(x)-u(x)|$ for all $x\in X$. Therefore, as $v_j\to u$ in $L^2(X)$, we have 
$u_j\to u$  in $L^{2}(X)$.

To see the $N^{1,2}$-convergence, we use~\cite[Proposition~6.3.23]{HKSTbook}.  
Using this proposition and the fact that $g_w\le \chi_\Om$, we get  
    \[
        g_{u_j-u}\le  \chi_{\{|v_j-u|>w\}\cap\Om} + g_{v_j-u}\cdot  \chi_{\{|v_j-u|\leq w\}\cap\Om}. 
    \]
Note that for any $\epsilon>0$,
    \[
        \int_X g_{u_j-u}^2\, d\mu  = \int_\Omega g_{u_j-u}^2 d\mu 
        =  \int_{\Omega_\epsilon} g_{u_j-u}^2 d\mu + \int_{\Omega\setminus \Omega_\epsilon} g_{u_j-u}^2 d\mu,
    \]
where $\Omega_\epsilon = \{x\in\Omega: \dist(x,X\setminus \Omega)>\epsilon\}$. So we get 
 \[
 \int_X g_{u_j-u}^2\, d\mu \leq 
 \mu(\{|v_j-u|>\epsilon\}\cap\Om_\epsilon) + \int_{\{|v_j-u|\le w\}\cap\Om} g_{v_j-u}^2 d\mu + \mu(\Om\setminus\Om_\epsilon).
 \]
 Note that
 \[
 \mu(\{|v_j-u|>\epsilon\}\cap\Om_\epsilon)\le \frac{1}{\epsilon^2}\int_{\Om_\epsilon}|v_j-u|^2\, d\mu.
 \]
 Since $v_j\to u$ in $L^2(X)$, it follows that $\limsup_{j\to\infty} \mu(\{|v_j-u|>\epsilon\}\cap\Om_\epsilon)=0$,
 and moreover, as $v_j\to u$ in the Hilbert space $N^{1,2}(X)$, it also follows that 
 $\limsup_{j\to\infty}\int_\Om g_{v_j-u}^2\, d\mu=0$. Therefore we have
    \[
        \limsup_{j\to \infty} \int_X g_{u_j-u}^2\, d\mu \leq \mu(\Om\setminus\Om_\epsilon).
    \]
As $\epsilon>0$ can be chosen arbitrarily small, and  $\bigcap_{\epsilon>0}\Om\setminus\Om_\epsilon$ is empty, 
it follows that $\limsup_{\epsilon\to 0^+}\mu(\Om\setminus\Om_\epsilon)=0$, and so
 \[
        \lim_{j\to \infty} \int_X g_{u_j-u}^2\,  d\mu  = 0. 
  \] 
Thus we have that $u_j\to u$ in $N^{1,2}(X)$.

Finally, note that $|\EC_{r_j}^{1/2} [u_j]-\EC_{r_j}^{1/2}[v_j]|\le \EC_{r_j}^{1/2}[u_j-v_j]$. By
Theorem~\ref{thm:Dirich-Lowerbound} we therefore have that
\[
|\EC_{r_j}^{1/2} [u_j]-\EC_{r_j}^{1/2}[v_j]|\le C\, \int_X g_{u_j-v_j}^2\, d\mu=C\, \int_\Om g_{v_j-u_j}^2\, d\mu\to 0
\text{ as }j\to\infty.
\]
It follows that $\EC[u]=\lim_{j\to\infty}\EC_{r_j}[v_j]=\lim_{j\to\infty}\EC_{r_j}[u_j]$ as desired.
\end{proof}

\section{Properties of $\EC$}\label{Sec:4}

In this section we will consider some key basic properties of the limit form $\EC$. 
There is a property of $\EC$ that would be essential in order to call $\EC$ a Dirichlet form; namely, that
$\EC$ is Markovian. In this section we will also prove that $\EC$ is Markovian and that it is local. 
The first lemma follows 
from~\cite[Proposition~11.9(e) and Theorem~11.10]{DalMaso}.

\begin{defn}\label{def:domain-dirich}
Let $\EC$ be a symmetric bilinear form from a Hilbert space $H$ to $[-\infty,\infty]$. Let 
$\mathcal{D}(\EC)$ be a dense linear subspace of $H$ such that $\EC$ is a symmetric bilinear form  
mapping $\mathcal{D}(\EC)$ to $\mathbb{R}$. Then we say $\mathcal{D}(\EC)$ is the \emph{domain} of $\EC$.
\end{defn}

\begin{defn}\label{def:core}
Let $C_c(X)$ be the collection of all real-valued continuous functions with compact support on $X$. A \emph{core} of a 
symmetric biliear form $\EC$ is a subset $\mathcal{C}$ of $\mathcal{D}(\EC) \cap C_c(X)$ such that $\mathcal{C}$ is 
dense in $\mathcal{D}(\EC)$ under the norm $\sqrt{(\cdot,\cdot)_H+\EC(\cdot)}$ and dense in $C_c(X)$ under the uniform norm.
\end{defn}

\begin{defn}\label{def:closed}
We say that $\EC$ is \emph{closed} if for any $(u_n)_n\subseteq \mathcal{D}(\EC)$ satisfying 
$(u_n-u_m,u_n-u_m)_H+\EC[u_n-u_m] \to 0$ as $m,n\to \infty$, there exists $u\in \mathcal{D}(\EC)$ such that 
$(u_n-u,u_n-u)_H+\EC[u_n-u] \to 0$ as $n\to \infty$. 
\end{defn}

In our setting, we have chosen the Hilbert space $H=N^{1,2}(X)$, and consequent
to Proposition~\ref{prop:Comparability} (or Theorem~\ref{thm:main1}), we also have that
$N^{1,2}(X)=\mathcal{D}(\EC)$. This is because the 
comparability of $\EC$ and the energy form on $N^{1,2}(X)$ (Proposition~\ref{prop:Comparability}) guarantees 
that $\EC$ should always be finite on $N^{1,2}(X)$.

We can choose the core $\mathcal{C}$ to be the collection of all Lipschitz functions on $X$. 
Since the measure $\mu$ on $X$ is doubling and supports a $2$-Poincar\'e inequality, it follows that Lipschitz
functions are dense in $N^{1,2}(X)$, see for instance~\cite[Section~8.2]{HKSTbook}.

\begin{lem}
$\EC$ is a symmetric bilinear form on $N^{1,2}(X)$. 
\end{lem}

The above lemma follows from~\cite[Proposition~11.9 and Theorem~11.10]{DalMaso}, where symmetric bilinear forms are called quadratic forms.

\begin{lem}\label{lem:loc_NpApprox}
The form $\EC$ is a local form on $N^{1,2}(X)$. Moreover, for each $u\in N^{1,2}(X)$ we have that
\[
\EC[u]=\lim_{k\to\infty}\EC_{r_k}[u].
\]
\end{lem}

\begin{proof}
We first prove the locality property of $\EC$.
Let $u,v\in N^{1,2}(X)$
such that the support $K_u$ of $u$ and the support $K_v$ of $v$ are disjoint. Then, as $X$ is compact, there is 
some $\rho>0$ such that $\text{dist}(K_u,K_v)\ge 10\rho$.
By the bilinearity and symmetry properties of $\EC$, we have $\EC[u+v]=\EC[u-v]+4\EC(u,v)$. 
To show that $\EC(u,v)=0$ (the definition of locality for a symmetric bilinear form), it suffices to show 
that $\EC[u+v]=\EC[u-v]$.
Note that $u+v,u-v$ are both zero in $X\setminus (K_u\cup K_v)$. Let $(\widehat{f_k})_k\subset N^{1,2}(X)$ be an
approximating sequence for $u+v$ as guaranteed by the $\Gamma$-convergence. Then, thanks to
Theorem~\ref{thm:approx-zeroTrace}, we can also assume that each $\widehat{f_k}=0$ on $X\setminus (K_u\cup K_v)$.
Then for each positive integer $k$, the functions $f_k=\widehat{f_k}\, \chi_{K_u}$ and $h_k=\widehat{f_k}\, \chi_{K_v}$
both belong to $N^{1,2}(X)$ and $\widehat{f_k}=f_k+h_k$.
Indeed, as there is a positive distance between $K_u$ and $K_v$, we can find two compactly supported
Lipschitz functions $\pip_u$ and $\pip_v$ so that $\pip_u=1$ on $K_u$, $\pip_v=1$ on $K_v$, and 
$f_k=\pip_u\widehat{f_k}$, $h_k=\pip_v\widehat{f_k}$. Thus $f_k, h_k\in N^{1,2}(X)$, as seen
from~\cite[Proof of Proposition~7.1.35]{HKSTbook}.  Moreover, $f_k\to u$ in $N^{1,2}(X)$ and
$h_k\to v$ in $N^{1,2}(X)$. It follows that $f_k-h_k\to u-v$ in $N^{1,2}(X)$. 
It follows that
\[
\EC[u-v]\le \liminf_{k\to\infty}\EC_{r_k}[f_k-h_k].
\]
On the other hand, when $r_k<\rho$, we have that 
\[
\EC_{r_k}[f_k-h_k]=\EC_{r_k}[f_k]+\EC_{r_k}[h_k]=\EC_{r_k}[f_k+h_k],
\]
and so as $\lim_{k\to\infty}\EC_{r_k}[\widehat{f_k}]=\EC[u+v]$, 
we have that 
\[
\EC[u-v]\le \EC[u+v].
\]
Replacing $v$ with $-v$ in the above argument also gives that $\EC[u+v]\le \EC[u-v]$, and so we have that
$\EC[u+v]=\EC[u-v]$. It follows that $\EC$ is a local form.

Next, for $u\in N^{1,2}(X)$, let $(u_k)_k$ be a sequence in $N^{1,2}(X)$ such that $u_k\to u$ in $N^{1,2}(X)$
and $\lim_{k\to\infty}\EC_{r_k}[u_k]=\EC[u]$; existence of such sequence is guaranteed by the definition
of $\Gamma$-convergence. Then
\[
\EC[u]^{1/2}\le \liminf_{k\to\infty}\EC_{r_k}[u]^{1/2}
\le \liminf_{k\to\infty}\left(\EC_{r_k}[u-u_k]^{1/2}+\EC_{r_k}[u_k]^{1/2}\right).
\]
By Theorem~\ref{thm:Dirich-Lowerbound} and by the fact that $u_k\to u$ in $N^{1,2}(X)$ we have
$\lim_{k\to\infty} \EC_{r_k}[u-u_k]=0$, see for example~\eqref{eq:not-bright}. 
It follows that
$\EC[u]\le \liminf_{k\to\infty}\EC_{r_k}[u_k]=\EC[u]$, from which 
the desired conclusion now follows.
\end{proof}

In \cite[page~372]{Mosco}, it is shown 
that the $\Gamma$-limit of bilinear forms on $L^2(X)$ is closed. In our setting, we consider a 
$\Gamma$-convergence with the ambient Hilbert space itself being $N^{1,2}(X)$. Therefore
for completeness of the exposition, we include our proof here.

\begin{lem}\label{lem:closed}
The symmetric bilinear form $\EC$ is closed.
\end{lem}

\begin{proof}
We first show that if $(f_n)_n$ is a sequence in $N^{1,2}(X)$ and $f\in L^2(X)$ such that $f_n\to f$ in $L^2(X)$ and
$\EC[f_n-f_m]\to 0$ as $m,n\to\infty$, then by modifying $f$ on a set of $\mu$-measure zero on $X$ if necessary, we have that
$f\in N^{1,2}(X)$ and $\Vert f_n-f_m\Vert_{N^{1,2}(X)}\to 0$ as $n,m\to\infty$. Indeed, as $\EC[f_n-f_m]\to 0$ as $m,n\to\infty$,
it follows from Theorem~\ref{thm:Dirich-Lowerbound} (or more specifically, from its limit version given in (2)~of Theorem~\ref{thm:main1}),
we have that $\int_Xg_{f_n-f_m}^2\, d\mu\to0$ as $n,m\to\infty$. It follows that $(f_n)_n$ is Cauchy in the Banach space $N^{1,2}(X)$,
and so such modification of $f$ exists.

Now let $(f_n)_n\subset \mathcal{D}(\EC)$ be a sequence such that $\Vert f_n-f_m\Vert_{N^{1,2}(X)}+\EC[f_n-f_m] \to 0$ as $m,n\to \infty$. 
Then necessarily $(f_n)_n$ is a Cauchy sequence in $N^{1,2}(X)$, and as $N^{1,2}(X)$ is a Banach space (see~\cite{Sh-Rev}),
there exists $f\in N^{1,2}(X)=\mathcal{D}(\EC)$ such that $\Vert f_n-f\Vert_{N^{1,2}(X)}\to 0$ as $n\to \infty$. By
Lemma~\ref{lem:loc_NpApprox}, we know that $f\in\mathcal{D}(\EC)$.
To complete the proof of closedness of $\EC$ we now show that $\EC[f_n-f] \to 0$ as $n\to \infty$. 

By~\cite[Proposition~6.8]{DalMaso}, we know that the $\Gamma$-limit is lower semicontinuous. 
Therefore, for each fixed $n$, using the lower semicontinuity of $\EC$, we have
\begin{align*}
\EC[f_n-f] \le \liminf_{m\to \infty} \EC[f_n-f_m].
\end{align*}
Hence, we have
\begin{align} \label{eqn.LimsupIs0}
0 \le \limsup_{n\to\infty} \EC[f_n-f] \le \limsup_{n\to\infty}\liminf_{m\to \infty} \EC[f_n-f_m]
= \lim_{m,n\to \infty}\EC[f_n-f_m] =0.
\end{align}
\end{proof}

\begin{prop}\label{prop:Markov}
The bilinear form $\EC$ is Markovian.
\end{prop}

\begin{proof}
To show that the form is Markovian, it suffices to show that if $u\in N^{1,2}(X)$, setting $v=\max\{0,\min\{1,u\}\}$,
we have $\EC[v]\le \EC[u]$. Since $(X,d,\mu)$ is doubling and supports a $2$-Poincar\'e inequality, we know that
Lipschitz functions form a dense subclass of $N^{1,2}(X)$; hence it suffices to prove this property for the case that
$u$ is Lipschitz continuous.

So let $u\in N^{1,2}(X)$ be Lipschitz on $X$, and set $E_-=\{u<0\}$, $E_+=\{u>1\}$, and $E_0=\{0\le u\le 1$.
Then as $X$ is compact, and as $\overline{E_-}$ and $\overline{E_+}$ are disjoint, we have that 
$\rho:=\text{dist}(\overline{E_-},\ \overline{E_+})>0$. So if $0<r<\rho/16$, then we have that for $\xB\in X_r$
and $\yB\in X_r$ with $\yB\sim\xB$,
the set $B(\xB,r/4)\cup B(\yB,r/4)$ cannot intersect both $\overline{E_+}$ and $\overline{E_-}$.

Now let $\xB,\yB\in X_r$ with $\xB\sim\yB$. If $B(\xB,r/4)\cup B(\yB,r/4)\subset\overline{E_-}$
or $B(\xB,r/4)\cup B(\yB,r/4)\subset \overline{E_+}$ or $B(\xB,r/4)\cup B(\yB,r/4)\subset E_0$, then
we have $|v_r(\xB)-v_r(\yB)|\le |u_r(\xB)-u_r(\yB)|$. If either 
$B(\xB,r/4)\cup B(\yB,r/4)$ intersects both $E_-$ and $E_0$ (the case that
$B(\xB,r/4)\cup B(\yB,r/4)$ intersects both $E_+$ and $E_0$ will be handled in s very similar manner).
In this case, we have
\[
|v_r(\xB)-v_r(\yB)|\le |v_r(\xB)-u_r(\xB)|+|u_r(\xB)-u_r(\yB)|+|u_r(\yB)-v_r(\yB)|.
\]
For such $\xB$, we have to estimate $|v_r(\xB)-u_r(\xB)|$ and $|v_r(\yB)-u_r(\yB)|$. Note that
\begin{align*}
|v_r(\xB)-u_r(\xB)|^2&=\bigg\vert \jint_{B(\xB,r/4)}(u-v)\, d\mu\bigg\vert^2\le \frac{1}{\mu(B(\xB,r/4))}\int_{B(\xB,r/4)\cap E_-}|u|^2\, d\mu\\
|v_r(\yB)-u_r(\yB)|^2&=\bigg\vert \jint_{B(\yB,r/4)}(u-v)\, d\mu\bigg\vert^2\le \frac{1}{\mu(B(\yB,r/4))}\int_{B(\yB,r/4)\cap E_-}|u|^2\, d\mu.
\end{align*}
Setting $L$ to be a Lipschitz constant of $u$ and that $B(\xB,r/4)\cup B(\yB,r/4)$ intersects
$E_-$ and $E_0$, and $d(\xB,\yB)\le 4r$, necessarily $|u|\le 5Lr$ on $B(\xB,r/4)\cup B(\yB,r/4)$.
Hence
\[
|v_r(\xB)-u_r(\xB)|^2\le 25L^2r^2 \ \text{ and }\ |v_r(\yB)-u_r(\yB)|^2\le 25L^2r^2.
\]
Similar estimates hold if $B(\xB,r/4)\cup B(\yB,r/4)$ intersects both $E_+$ and $E_0$, now with $|u|^2$ replaced by
$|u-1|^2$. Note also that when $B(\xB,r/4)\cup B(\yB,r/4)$ intersects both $E_-$ and $E_0$ and
$\xB\sim\yB$, necessarily we must have
$B(\xB,r/4)\cup B(\yB,r/4)\subset\bigcup_{z\in \partial E_-}B(z, 6r)$ because $X$ is a geodesic space and $u$ is 
continuous. We set
\[
(\partial E_-)_r=\bigcup_{z\in \partial E_-}B(z, 6r),\qquad 
\text{ and }\qquad 
(\partial E_+)_r=\bigcup_{z\in \partial E_+}B(z, 6r), 
\]
and let
\[
X_r[v]:=\{\xB\in X_r\, :\, B(\xB,r/4)\subset(\partial E_-)_r\cup(\partial E_+)_r\}.
\]
Note that
\begin{align*}
\EC_r[v]&=\sum_{\xB\in X_r}\sum_{\yB\sim\xB}\frac{|v_r(\xB)-v_r(\yB)|^2}{r^2}\, \mu_r(\xB)\\
&
\le\sum_{\xB\in X_r[v]}\sum_{\yB\sim\xB}
\frac{\left(|v_r(\xB)-u_r(\xB)|+|u_r(\xB)-u_r(\yB)|+|u_r(\yB)-v_r(\yB)|\right)^2}{r^2}\, \mu_r(\xB)\\
&\qquad\qquad\qquad
+\sum_{\xB\in X_r\setminus X_r[v]}\sum_{\yB\sim\xB}\frac{|u_r(\xB)-u_r(\yB)|^2}{r^2}\, \mu_r(\xB).
\end{align*}
It follows that 
\begin{align*}
\EC_r[v]&\le
\sum_{\xB\in X_r}\sum_{\yB\sim\xB}\frac{|u_r(\xB)-u_r(\yB)|^2}{r^2}\, \mu_r(\xB)\\
&\qquad\qquad
+\sum_{\xB\in X_r[v]}\sum_{\yB\sim\xB}\frac{|v_r(\xB)-u_r(\xB)|^2}{r^2}\, \mu_r(\xB)+\sum_{\xB\in X_r[v]}\sum_{\yB\sim\xB}\frac{|v_r(\yB)-u_r(\yB)|^2}{r^2}\, \mu_r(\xB)\\
&\qquad\qquad\qquad
+2\sum_{\xB\in X_r[v]}\sum_{\yB\sim\xB}\frac{|v_r(\xB)-u_r(\xB)|\cdot |u_r(\xB)-u_r(\yB)|}{r^2}\, \mu_r(\xB)\\
&\qquad\qquad\qquad
+2\sum_{\xB\in X_r}\sum_{\yB\sim\xB}\frac{|v_r(\yB)-u_r(\yB)|\cdot |u_r(\xB)-u_r(\yB)|}{r^2}\, \mu_r(\xB)\\
&\qquad\qquad\qquad
+2\sum_{\xB\in X_r[v]}\sum_{\yB\sim\xB}\frac{|v_r(\xB)-u_r(\xB)|\cdot |v_r(\yB)-u_r(\yB)|}{r^2}\, \mu_r(\xB).
\end{align*}
Using H\"older's inequality, we have
\begin{align*}
\EC_r[v] & \le \EC_r[u]+\sum_{\xB\in X_r[v]}\sum_{\yB\sim\xB}\frac{|v_r(\xB)-u_r(\xB)|^2}{r^2}\, \mu_r(\xB)
+\sum_{\xB\in X_r[v]}\sum_{\yB\sim\xB}\frac{|v_r(\yB)-u_r(\yB)|^2}{r^2}\, \mu_r(\xB)
\\
&\qquad\qquad
+2\EC_r[u]^{1/2} \cdot \left(\sum_{\xB\in X_r[v]}\sum_{\yB\sim\xB}\frac{|v_r(\xB)-u_r(\xB)|^2}{r^2}\, \mu_r(\xB)\right)^{1/2}
\\
&\qquad\qquad
+2\EC_r[u]^{1/2} \cdot \left(\sum_{\xB\in X_r[v]}\sum_{\yB\sim\xB}\frac{|v_r(\yB)-u_r(\yB)|^2}{r^2}\, \mu_r(\xB)\right)^{1/2}
\\
&\qquad\qquad
+2 \left[\left(\sum_{\xB\in X_r[v]}\sum_{\yB\sim\xB}\frac{|v_r(\xB)-u_r(\xB)|^2}{r^2}\, \mu_r(\xB)\right)\left(\sum_{\xB\in X_r[v]}\sum_{\yB\sim\xB}\frac{|v_r(\yB)-u_r(\yB)|^2}{r^2}\, \mu_r(\xB)\right)\right]^{1/2}.
\end{align*}
It now suffices to show that both the term $\sum_{\xB\in X_r[v]}\sum_{\yB\sim\xB}\frac{|v_r(\yB)-u_r(\yB)|^2}{r^2}\, \mu_r(\xB)$ 
and the term $ \sum_{\xB\in X_r[v]}\sum_{\yB\sim\xB}\frac{|v_r(\yB)-u_r(\yB)|^2}{r^2}\, \mu_r(\xB)$ tend to $0$ as $r\to 0$.
From the discussion prior to the above estimates, we see that
\begin{align*}
 \sum_{\xB\in X_r[v]}\sum_{\yB\sim\xB}\frac{|v_r(\xB)-u_r(\xB)|^2}{r^2}\, \mu_r(\xB)
&\le N\,  \sum_{\xB\in X_r[v]}\frac{|v_r(\xB)-u_r(\xB)|^2}{r^2}\, \mu_r(\xB)\\
&\le N\, 25L^2\, \sum_{\xB\in X_r[v]}\mu_r(\xB)\\
&\le C\, 25L^2\, \mu((\partial E_-)_r\cup(\partial E_+)_r),
\end{align*}
where we have used the fact that for each $\xB\in X_r$ there are at most $N$ number of vertices $\yB\in X_r$ such that
$\xB\sim\yB$, with $N$ depending only on the doubling constant associated with $\mu$; and we also used
fact that there is at most 
bounded overlap property of the balls $B(\yB,r)$, $\yB\in X_r$, to obtain the last inequality.
A similar argument also gives
\[
 \sum_{\xB\in X_r[v]}\sum_{\yB\sim\xB}\frac{|v_r(\yB)-u_r(\yB)|^2}{r^2}\, \mu_r(\xB)
\le C\, 25L^2\, \mu((\partial E_-)_r\cup(\partial E_+)_r).
\]
It follows that
\[
\EC_r[v]\le \EC_r[u]+10C\, L\, \EC_r[u]^{1/2}\mu((\partial E_-)_r\cup(\partial E_+)_r)^{1/2}
+25\, C\, L^2\, \mu((\partial E_-)_r\cup(\partial E_+)_r).
\]
 So if 
$\mu(\partial E_-)+\mu(\partial E_+)=0$, then we have that
\[
\EC[v]\le \liminf_{k\to\infty}\EC_{r_k}[v]\le \liminf_{k\to\infty}\EC_{r_k}[u]=\EC[u].
\]
The last equality follows from Lemma~\ref{lem:loc_NpApprox} above.
The assumption of $\mu(\partial E_-)+\mu(\partial E_+)=0$ is not a serious obstacle, as we know that for 
almost every $t\in \R$ we have that $\mu(\{u=t\})=0$, and so we can approximate $0$ from below by such $t_-<0$,
and approximate $1$ from above by such $t_+>1$, and consider $v_{t_-,t_+}=\max\{t_-,\min\{t_+,u\}\}$
and obtain the Markov property for these, and then let $t_-\to 0$, $t_+\to 1$ and note that 
$v_{t_-,t_+}\to v$ in $N^{1,2}(X)$; we know that $\EC$ is bounded
(from Theorem~\ref{thm:Dirich-Lowerbound} above)  bilinear form on $N^{1,2}(X)$.
The fact that $\EC_r[v_{t_-,t_+}]$ converges to $\EC_r[v]$ can be obtained from 
the closability of $\EC$ (lemma~\ref{lem:closed} above) and  Theorem~\ref{thm:Dirich-Lowerbound}, together with the fact that
$v_{t_-,t_+}\to v$ in $N^{1,2}(X)$ as $t_-\to 0$ and $t_+\to 1$.
\end{proof}

\bibliography{references-Updated20221030}
\bibliographystyle{amsplain}

\end{document}